\documentclass[10pt,reqno]{amsart} 

\usepackage{amssymb,latexsym,  
amscd, 
amsthm, 
graphicx, epsfig
}
\usepackage{cite} 

\usepackage[height=190mm,width=130mm]{geometry} 

\theoremstyle{plain}

\newtheorem{lemma}{Lemma}
\newtheorem{corollary}{Corollary}
\newtheorem{proposition}{Proposition}

\theoremstyle{definition}

\theoremstyle{remark}

\def\1{{\bf 1}}
\newcommand{\nc}{\newcommand}

\newcommand{\Z}{\mathbb{Z}}
\newcommand{\C}{\mathbb{C}}

\numberwithin{equation}{section} 

\newcommand{\nn}{\nonumber \\}

\newcommand{\wt}{\mbox{\rm wt}\ }

\newcommand{\N}{\mathbb{N}}
\newcommand{\M}{\mathcal{M}}
\newcommand{\F}{\mathcal{G}}
\newcommand{\Po}{\mathcal P}

\newcommand{\one}{\mathbf{1}}

\begin{document}
\title[The extension of cochain complexes of meromorphic functions]     
{The extension of cochain complexes of meromorphic functions to multiplications}   
     
\author{Daniel Levin${}^\&$}
\address{ ${}^\&$ Mathematical Institute \\ Uiversity of Oxford \\ Andrew Wiles Building \\
 Radcliffe Observatory Quarter (550) \\ Woodstock Road \\ Oxford \\ OX2 6GG \\ United Kingdom  }     
\author{Alexander Zuevsky${}^*$}
\address{ ${}^*$ Institute of Mathematics \\ Czech Academy of Sciences\\ Zitna 25, 11567 \\ Prague\\ Czech Republic}

\email{levindanie@gmail.com}

\email{zuevsky@yahoo.com}





\begin{abstract}
 Let $\mathfrak g$ be an infinite-dimensional Lie algebra and 
 $G$ be the algebraic completion of its module.  
Using a geometric interpretation in terms of sewing two Riemann spheres with 
a number of marked points, 
we introduce a 
multiplication between elements of two spaces $\M^k_m(\mathfrak g, G)$
and $\M^n_{m'}(\mathfrak g, G)$ of meromorphic functions 
depending on a number of formal complex parameters $(x_1, \ldots, x_k)$ and 
$(y_1, \ldots, y_n)$  
  with specific analytic and symmetry properties, and 
associated to $\mathfrak g$-valued series. 
These spaces form a chain-cochain complex 
with respect to a boundary-coboundary operator. 
The main result of the paper shows that 
 the multiplication is defined by an   
absolutely convergent series and 
 takes values in the space $\M^{k+n}_{m+m'}(\mathfrak g, G)$.  
\end{abstract}

\vskip12pt  

\maketitle
\section{Introduction: Motivation and geometrical interpretation}
\label{introduction}
For purposes of construction of non-trivial cohomology classes 
for an infinite-dimensional Lie algebra $\mathfrak g$ \cite{K} 
on manifolds 
it is important to define a multiplication of elements of cochain complex spaces of meromorphic 
functions with predetermined analytic properties and depending on $\mathfrak g$-series.  
Predetermined functions can be parameterized by formal complex parameters associated to local coordinates 
of marked points on Riemann spheres. 
In this paper we introduce the multiplication of elements 
of double cochain complex $(\M^n_m(\mathfrak g, G)$, $\delta^n_m)$-spaces by 
 involving the geometrical procedure \cite{Y} of sewing two Riemann spheres. 
Such multiplication is then  
parameterized by a nonzero complex number $\epsilon$ which is the complex parameter of 
the Riemann spheres sewing.  
For two chain complex spaces  
 $\M^k_m(\mathfrak g, G)$ and $\M^n_{m'}(\mathfrak g, G)$, 
 we associate formal complex parameters  
$(x_1, \ldots, x_k)$ and $(y_1, \ldots, y_n)$ 
to local coordinates vanishing on $k$ and $n$ marked points on two Riemann spheres. 
The sewing brings about another Riemann sphere,  
and formal complex parameters of predetermined meromorphic functions 
of the space 
$\M^{k+n}_{m+m'}$ are identified with 
parameters of the resulting sphere.  
The problem of defining a multiplication of elements of cochain complex spaces 
is very important for cohomological problems in conformal field theory \cite{FMS, TUY}, 
infinite-dimensional Lie algebras \cite{K},  
the theory of integrable models, 
 as well as for further applications to 
cohomologies of smooth manifolds \cite{Fuks1}.   
The plan of the paper is the following. 
In Section \ref{valuedef}, for an infinite-dimensional Lie algebra, 
we describe the axiomatics of meromoprhic functions with predetermined analytic properties.  
The cochain complex spaces $\M^._.$ 
of predetermined meromorphic functions are then defined. 
In Section \ref{application}
we describe first the geometric procedure of forming a Riemann sphere by sewing 
two initial Riemann spheres. 
A multiplication for elements of two $\M^._.$-spaces is introduced. 
The main result of this paper, namely, Proposition \ref{tolsto} showing that the product of 
 elements of spaces $\M^k_m$ and $\M^n_{m'}$ belongs to $\M^{k+n}_{m+m'}$ is proven.  
In particular, we show that functions obtained as a result of multiplication of elements of 
spaces of predetermined meromorphic functions are absolutely convergent meromorphic functions 
with predetermined analytic properties. 
\section{The cochain complex spaces $\M^n_m(\mathfrak g, G)$}   
\label{valuedef}
Let $\mathfrak g$ be an infinite-dimensional Lie algebra, and $W$ its module. 
Denote by $F_n \mathbb C$ the  
configuration space of $n \ge 1$ ordered formal complex parameters in $\mathbb C^n$, 
 $F_n\mathbb C=\{(z_1, \ldots, z_n) \in \mathbb C^n\;|\; z_i \ne z_j, i\ne j\}$. 
Denote by $G=G_{(z_1, \ldots, z_n)}$ 
 the graded (with respect to a grading operator $K$)   
algebraic completion of $W$.  
It is assumed that on $G$ there exists a non-degenerate bilinear pairing $(.,.)$.  
For $z=z_j$, $1 \le j \le n$, the element $T= - \partial_z$     
acts as the translation operator, and 
$K = - z \partial_z$ acts as a grading operator. 
 Let $G'=\coprod_{\l \in \Z} G_{\l}^*$ denotes 
the graded dual for $G=\bigoplus_{\lambda \in{\Z} }G_{(\lambda)}$ 
 with respect to $(.,.)$.
 Assume that 
 $G$ is equipped with a map    
$\gamma_g:  G_z  \to   G[[z, z^{-1}]]$,  as a formal series  
 $g \mapsto \gamma_g(z) \equiv \sum_{l\in \Z} g_l \; z^l$ for $z\in \C$. 
We assume that 
$G_{ (\l) }=\{w\in G | K w=\l w, \; \l=\wt(w)\}$.   
 Moreover we require that
$\dim G_{(\l)} < \infty$, i.e., 
 it is finite, and for fixed $\l$, $G_{(n+\l)}=0$,  for all 
small enough integers $n$. 
For $g$, $w\in G$,    
 $n \in \Z$,   
$g_n w=0$, $n \gg 0$, 
 $\gamma_{\bf 1}(z)={\rm Id}$. 
 For $g \in G$, $\gamma_g(z) w$ contains only finitely many negative 
power terms, that is, $\gamma_g(z)w\in G$.  
By $\one_G$ we denote the highest weight element in $G$.  
We normalize the pairing by the condition $(\overline{\one}_G, \one_G)=1$. 

We now define the space of meromorphic functions $\F(g_1, z_1; \ldots; g_n, z_n)$
 depending on $n$ 
 $G$-elements and $n$ formal complex parameters.  
For $\F$ we allow poles only at 
$z_i=z_j$, $i\ne j$.  
We define left action of the permutation group $S_n$ on $\F(g_1, z_1; \ldots; g_n, z_n)$ 
by
$\nc{\bfzq}{(z_1, \ldots, z_n)}
\sigma(\F)(g_1, z_1; \ldots; g_n, z_n)=\F\left(g_1, z_{\sigma(1)};  \ldots;  g_n, z_{\sigma(n)} \right)$.  
 For $g_1$, $g_2 \in G$, $w \in G$, we require for $\F$ that
 the functions   
$\F (\gamma_{g_1}(z_1) \; \gamma_{g_2}(z_2)w)$,     
 $\F( \gamma_{g_2}(z_2) \; \gamma_{g_1}(z_1) w)$,  and 
$\F( \gamma_{ \gamma_{g_1}(z_1-z_2)  g_2 } (z_2) w )$,    
are absolutely convergent
in the regions $|z_1|>|z_2|>0$, $|z_2|>|z_1|>0$,
$|z_2|>|z_1-z_2|>0$, 
respectively, to a common function 
in $z_1$ and $z_2$. 
The poles are only allowed at $z_1=0=z_2$, $z_1=z_2$.
 If $g$ is homogeneous then 
$g_m G_{(n)}\subset G_{(\wt u-m-1+n)}$.   
For a subgroup  ${\mathfrak G} \subset {\rm Aut} \; G$,   
 $\mathfrak G$ acts on $G$ as automorphisms if 
$g \; \gamma_h(z) \; g^{-1}=\gamma_{gh}(z)$,  
for all $g$, $h \in {\mathfrak G}$.    
The operator $K$ satisfies the derivation property
$\gamma_{K g}(z)=\frac{d}{dz}\gamma_g(z)$. 
Denote by $T_i.$ the operator acting on the $i$-th entry.
 We then define the action of partial derivatives on an element $\F(g_1, z_1;  \ldots; g_n, z_n)$  
\begin{eqnarray}
\label{cond1}
\partial_{z_i} \F(g_1, z_1; \ldots;  g_n, z_n)  &=& \F(T_i. \; (g_1, z_1; \ldots; g_n, z_n)),   
\nn
\sum\limits_{i \ge 1} \partial_{z_i}  \F(g_1, z_1; \ldots;  g_n, z_n)   
&=&  T \F(g_1, z_1; \ldots;  g_n, z_n).  
\end{eqnarray}
For $z \in \C$,  let 
\begin{eqnarray}
\label{ldir1}  
 e^{zT} \F (g_1, z_1; \ldots;  g_n, z_n)   
 = \F(g_1, z_1+z ; \ldots;  g_n, z_n+z).   
\end{eqnarray}
 Let us denote by 
${\rm Ins}_i(A)$ the operator of multiplication by $A \in \C$ at the $i$-th position. 
Then we assume that both sides of the expression 
\[
\F \left((g_1, \ldots, g_n) \right., {\rm Ins}_i(z_1, \ldots, z_n) \; \left.(z_1, \ldots, z_n) \right) =   
 \F\left( {\rm Ins}_i (e^{zT}) \right. \; \left. (g_1, z_1; \ldots;  g_n, z_n)\right),     
\]
are absolutely convergent
on the open disk $|z|<\min_{i\ne j}\{|z_i-z_j|\}$, and 
 equal as power series expansions in $z$. 
 For $z\in \C^{\times}$, 
$(zz_1, \ldots, zz_n) \in F_n\C$, we require for functions  
\begin{eqnarray}
\label{loconj}
z^K \F (g_1, z_1; \ldots;  g_n, z_n) =  
 \F \left(z^K g_1, zz_1; \ldots; z^K g_n, zz_n \right). 
\end{eqnarray}
For an arbitrary fixed $\theta \in G'$,      
 a map linear in $(g_1, \ldots, g_n)$ and $(z_1, \ldots, z_n)$, 
$\F: (z_1;  \ldots;  z_n) 
\mapsto   
     (\theta, f(g_1, z_1; \ldots; g_n, z_n ))$, 
delivers a particular example of 
  a meromorphic function in $(z_1, \ldots, z_n)$ 
which depends on $(g_1, \ldots, g_n)$. 

Now we recall further conditions on meromorphic functions associated to 
 a number of $\gamma_G$-series.  We call such functions predetermined combined with 
a number of $G$-series 
on a domain.   
By this we mean functions with specific 
 analytical behavior taking into account of Lie-algebra series.  
We denote by $\Po_k: G \to G_{(k)}$, $k \in \Z$,      
the projection of $G$ on $G_{(k)}$.
  Following \cite{Huang}, we formulate the following definitions and propositions. 
For $i=1, \dots, (l+k)n$, $k \ge 0$,  $ 1 \le l', l'' \le n$, 
 let $(l_1, \ldots, l_n)$ be a partition of $(l+ k)n     
=\sum\limits_{i \ge 1} l_i$, and $k_i= \sum_{j=1}^{i-1} l_j$.  
For $\zeta_i \in \C$,  
define 
$H_i  =\F ( \gamma_{  g_{k_1+ 1} }( z_{k_1+ 1}  - \zeta_i) \ldots 
\gamma_{  g_{k_i+l_i} }( z_{k_i+l_i}  - \zeta_i)\one_G)$, 
for $i=1, \dots, n$.
 It is assumed that  
the function   
 $\sum\limits_{ (r_1, \ldots, r_l) \in \Z^l}  
\F(  \Po_{r_1} H_1, \zeta_1; \ldots;  \Po_{r_l} H_l, \zeta_l )$,     
is absolutely convergent to an analytically extension 
in $(z_1, \ldots, z_n)_{l+k}$ 
in the domains 
$|z_{k_i+p} -\zeta_i| 
+ |z_{k_j+q}-\zeta_j|< |\zeta_i -\zeta_j|$,    
for $i$, $j=1, \dots, k$, $i\ne j$, and for $p=1, 
\dots$,  $l_i$, $q=1$, $\dots$, $l_j$.  
The convergence and analytic extension do not depend on complex parameters $\zeta_l$. 
On the diagonal of $(z_1, \ldots, z_n)_{l+k}$   
 the order of poles is bounded from above by described positive numbers 
$\beta(g_{l',i}, g_{l'', j})$.   
 For $(g_1, \ldots, g_{l+k}) \in G$,
$z_i\ne z_j$, $i\ne j$ 
$|z_i|>|z_s|>0$, for $i=1, \dots, k$, 
$s=k+1, \dots, l+k$ the sum 
$\sum_{q\in \Z}$  
$\F(  \gamma_{g_1}(x_1) \ldots  \gamma_{g_k}(x_k)  
 \Po_q ( \gamma_{g_{1+k}}(z_{1+k}) \ldots \gamma_{g_{l+k}}(x_{l+k}) )\one_G)$,     
is absolutely convergent and 
 analytically extendable to a
 function 
in variables $(z_1, \ldots, z_n)_{l+k}$. 
The order of pole that is allowed at 
$z_i=z_j$ is bounded from above by the numbers 
$\beta(g_{l', i}, g_{l'', j})$. 

Let  $S_l$ be the permutation group. 
For $l \in \N$ and $1\le s \le l-1$, let $J_{l; s}$ be the set of elements of
$S_l$ which preserve the order of the first $s$ numbers and the order of the last 
$l-s$ numbers, that is,
$J_{l, s}=\{\sigma\in S_l\;|\;\sigma(1)<\ldots <\sigma(s),\;
\sigma(s+1)<\ldots <\sigma(l)\}$. 
The elements of $J_{l; s}$ are called shuffles, and we use the notation 
 $J_{l; s}^{-1}=\{\sigma\;|\; \sigma\in J_{l; s}\}$.
 For any set of $G$-elements $g_i$, $g_j \in G$,   
and formal complex parameters $z_i$, $z_j$, $1\le i \le n$, 
$1\le j \le m$, 
 $n \ge 0$, $m \ge 0$,    
we denote by $\M^n_m(\mathfrak g, G)$,      
the space of all predetermined meromorphic functions combined with  
  $m$ $\gamma$-series  
 $(\F(z_1, \ldots,  z_n), \gamma_{g_j}(z'_j))$       
satisfying \eqref{ldir1}, \eqref{loconj}, 
the conditions of the previous paragraph,
and the symmetry property 
with respect to action of the symmetric group $S_n$: 
\begin{equation}
\label{shushu} 
\sum_{\sigma\in J_{n; s}^{-1}}(-1)^{|\sigma|}
\F\left(g_{\sigma(1)}, z_{\sigma(1)}; \ldots;  g_{\sigma(n)}, z_{\sigma(n)} \right) =0.
\end{equation}
For fixed $\mathfrak g$ and $G$ we will use the notation $\M^n_m$. 

One defines 
 $\F * ((f(g_1, z_1; \ldots;    
 g_{l_1}, z_{l_1}; \one_G, z);   
 \ldots;
  f(g_1^{'\ldots'}, z_1^{'\ldots'}; \ldots;    
 g_{l_n}^{'\ldots'}, z_{l_n}^{'\ldots'}; \one_G, z)):   
 G\to \C$:  
\begin{eqnarray*}
 && \F* ( f(g_1, z_1; \ldots;    
 g_{l_1}, z_{l_1}; \one_G, z);
 \ldots; 
 f(g_1^{'\ldots'}, z_1^{'\ldots'}; \ldots;    
 g_{l_n}^{'\ldots'}, z_{l_n}^{'\ldots'}; \one_G, z))
\nn
&& \qquad = \F ( f(g_{l_1}, z_{l_1}; \one_G, z);  
 \ldots;
f(g_1^{'\ldots'}, z_1^{'\ldots'};     
 g_{l_n}^{'\ldots'}, z_{l_n}^{'\ldots'}; \one_G, z)). 
\end{eqnarray*}
The action 
$\F(g_1, z_1; \ldots; g_m, z_m)*_0$ $f'(g_{m+1}$, $z_{m+1}$; $\ldots$; 
$g_{m+n}$, $z_{m+n})$ : $G$ $\to$ $G_{z_1, \dots, z_{m+n-1}}$,
 is given by 
\begin{eqnarray*}
&& \F(g_1, z_1; \ldots; g_m, z_m)*_0 f'(g_{m+1}, z_{m+1}; \ldots; g_{m+n}, z_{m+n}) 
\nn
&& \qquad = \F( f(g_1, z_1; \ldots; g_m, z_m);   
 f'(g_{m+1}, z_{m+1}; \ldots; g_{m+n}, z_{m+n})). 
\end{eqnarray*}
We introduce also  
$\F(g_1, z_1; \ldots; g_m, z_m)$ $*_{m+1}$ 
$f'(g_{n+1}, z_{n+1};  \ldots; g_{n+m}, z_{n+m})$ : $G\to G_{z_1, \ldots, z_{m+n-1}}$, 
 defined by 
\begin{eqnarray*}
 && \F(g_1, z_1; \ldots; g_m, z_m)   *_{m+1} (g_{n+1}, z_{n+1};  \ldots; g_{n+m}, z_{n+m})) 
\nn
 && \qquad = \F( f(g'_1, z'_1; \ldots; g'_n, z'_n) 
; g_{n+1}, z_{n+1};  \ldots; g_{n+m}, z_{n+m})).  
\end{eqnarray*}
The following result holds  
\begin{proposition}
\label{n-comm}
For $(g_1, \dots, g_n)\in G$, 
$\F(z_1, \ldots, z_n)$,  
is absolutely convergent in the region $|z_1|>\ldots >|z_n|>0$,   
to a function
in $(z_1, \dots, z_n)$ with the only possible poles at $z_i=z_j$, $i\ne j$,
and $z_i=0$.  
The $\F$ is invariant with respect to the action of  
$\sigma\in S_n$ on all entries $(z_1, \ldots, z_n)$. 
\end{proposition}
The following useful proposition holds:  
\begin{proposition}
\label{comp-assoc}
Let $\F: G \to \C$  
is a predetermined meromorphic function combined 
 with $m$ series $\gamma_{g_j}(y_j)$, $1 \le j \le m$, $m \ge 0$.
 Then
\begin{enumerate}
\item For $p\le m$, $\F$ is a predetermined function combined with
 $p$ $\gamma_{g_p}(x_p)$-series, and for 
$p$, $q\in \Z_{+}$ such that $p+q\le m$ and 
$(l_1, \dots, l_n) \in \Z_+$ such that $l_1+\ldots +l_n=p+n$, 
$\F * $(  $f (g_1', z_1'; \ldots; g_{l_1}', z_{l_1}'; \one_G, z'); $  
$\ldots;$ 
 $f (g_1^{'\ldots'}, z_1^{'\ldots'}$;  
$\ldots$; $g_{l_n}^{'\ldots'}, z_{l_n}^{'\ldots'}; \one_G, z^{'\ldots'}))$ 
and $\F(g_1, z_1; \ldots; g_n, z_n) *_{p+1}\F$  
are combined with $q$ series. 
\item For $p$, $q\in \Z_+$ such that $p+q\le m$,  
$(l_1, \dots, l_n) \in \Z_+$ such that $l_1+\ldots +l_n=p+n$ and
$k_1, \dots, k_{p+n} \in \Z_+$ such that $k_1+\ldots +k_{p+n}=q+p+n$,
we have
\begin{eqnarray*}
&& (\F* f(g_1', z_1'; \ldots; g_{l_1}', z_{l_1}'; \one_G, z');
 \ldots;
 f(g_1^{'\ldots'}, z_1^{'\ldots'}; \ldots; g_{l_n}^{'\ldots'}, z_{l_n}^{'\ldots'}; \one_G, z^{'\ldots'}))  
\nn
&& \quad *   
 f(g_1^*, z_1^*; \ldots; g_{k_1}^*, z_{k_1}^*; \one_G, z^*);   
  \ldots  
 f(g_1^{*\ldots*}, z_1^{*\ldots*}; \ldots; g_{k_{p+n}}^{*\ldots*},   
 z_{k_{p+n}}^{*\ldots*}; \one_G, z^{*\ldots*}))  
\nn
&& \quad = \F *   
f(g_1^{*\ldots*}, z_1^{*\ldots*}; \ldots;  
 g_{k_1+\ldots +k_{l_1} }^{*\ldots*}, z_{k_1+\ldots +k_{l_1} }^{*\ldots*}; \one_G, z^{*\ldots*});  
\nn
&& \quad 
\ldots 
f(g_1^{*\ldots*}, z_1^{*\ldots*};  \ldots;  
g_{k_{l_1+\ldots +l_{n-1}+1}+\ldots +k_{p+n})}^{*\ldots*},  
z_{k_{l_{1}+\ldots +l_{n-1}+1}+\ldots +k_{p+n}) }^{*\ldots*}; \one_G, z^{*\ldots*}).
\end{eqnarray*}
\item For $p$, $q\in \Z_+$ such that $p+q\le m$ and
$(l_1, \dots, l_n) \in \Z_+$ such that $\sum_{i=1}^n l_i=p+n$, 
we have
\begin{eqnarray*}
 &&\F(g_1, z_1; \ldots; g_q, z_q)*_{q+1}
 (\F * ( f (g_1', z_1'; \ldots; g_{l_1}', z_{l_1}'; \one_G, z'); 
 \ldots; 
\nn
 && \qquad f (g_1^{'\ldots'}, z_1^{'\ldots'};  
\ldots; g_{l_n}^{'\ldots'}, z_{l_n}^{'\ldots'}; \one_G, z^{'\ldots'}) )  
\nn
&& \quad
=(\F(g_1, z_1; \ldots; g_q, z_q) *_{q+1} \F) 
*(f (g_1', z_1'; \ldots; g_{l_1}', z_{l_1}'; \one_G, z');   
 \ldots; f(g_1^{'\ldots'}, z_1^{'\ldots'};   
\ldots;
\nn
&& \qquad 
 g_{l_n}^{'\ldots'}, z_{l_n}^{'\ldots'}; \one_G, z^{'\ldots'}) ).  
\end{eqnarray*}
\item For $p$, $q\in \Z_+$ such that $p+q\le m$, we have 
\begin{eqnarray*}
 &&\F (g_1, z_1; \ldots; g_p, z_p) *_{p+1}  
 (\F(g_1',z_1'; \ldots; g_q', z_q')*_{q+1} \F') 
\nn
&& \qquad = \F(g_1, z_1; \ldots; g_p, z_p; g'_1, z'_1; 
\ldots;  g_q', z_q'; g, z)  *_{p+q+1} \F'. 
\end{eqnarray*}
\end{enumerate}
\end{proposition}
\begin{proposition}
\label{correl-fn}
For $k$, $(l_1, \dots, l_{n+1}) \in \Z_+$ and 
$(g_1^{(1)}, \dots, g_{l_1}^{(1)}, \dots, g_1^{(n+1)}, \dots$,
$g_{l_{n+1}}^{(n+1)}) \in G$, the series 
\begin{eqnarray*}
&& \sum_{(r_1, \ldots, r_n) \in \Z, \; r_{n+1}\in \Z}
 \F (  \Po_{r_1}
 (f(g_1^{(1)}, z_1^{(1)};   
\ldots;  
 g_{l_1}^{(1)}, z_{l_1}^{(1)}; \one_G, z^{(0)}_1) ); \ldots;     
\nn
&& 
\qquad  \qquad 
 \Po_{r_n} (f(g_1^{(n)}, z_1^{(n)};  
 \ldots;  
  g_{l_n}^{(n)}, z_{l_n}^{(n)}; \one_G, z^{(0)}_n))   
\nn
 &&
\qquad  \qquad 
\Po_{r_{n+1}} (f(g_1^{(n+1)}, z_1^{(n+1)};  
\ldots;  
 g_{l_{n+1}}^{(n+1)}, z_{l_{n+1}}^{(n+1)};  g, z^{(0)}_{n+1})),    
\end{eqnarray*}
converges absolutely to  
\begin{eqnarray*}
 \F(g_1^{(1)}, z_1^{(1)} +z^{(0)}_1; \ldots;   
 g_{l_1}^{(1)}, z_{l_1}^{(1)}+z^{(0)}_1;  \ldots; 
 g_1^{(n+1)}, z_1^{(n+1)}+z^{(0)}_{n+1};   
g_{l_{n+1}}^{(n+1)},  z_{l_{n+1}}^{(n+1)}+z^{(0)}_{n+1}),  
\end{eqnarray*} 
when $0<|z_p^{(i)}| + |z_q^{(j)}|< |z^{(0)}_i -z^{(0)}_j|$ 
for $i$, $j=1, \dots, n+1$, $i\ne j$, $p=1, \dots,  l_i$, $q=1, \dots, l_j$.
\end{proposition}
Finally, we note  
\begin{lemma}
\label{obvlem}
\begin{eqnarray*}
&& \sum_{q\in \Z} \F (
g''_1, z_1; \ldots; 
g''_{m+m'}, z_{m+m'};   
\nn
&&
\qquad \qquad 
\Po_q (f(g''_{m+m'+1}, z_{m+m'+1}; \ldots; g''_{m+m'+k+n}, z_{m+m'+k+n})) 
%
\nn
&& =
\sum_{l \in \Z, \; g\in G_{(l)} }\epsilon^l
  \F (
g_{k+1}, x_{k+1}; \ldots;
g_{k+m}, x_{k+m}; 
\Po_q ( 
  f (g_1, x_1;  \ldots; g_k, x_k); g, \zeta_1) ) ) 
\nn
&&
 \qquad  \F(
 g'_{n+1}, y_{n+1}; \ldots;  
g'_{n+m'}, y_{n+m'};   
\Po_q(f 
(g'_1, y_1; \ldots; g'_n, y_n;  \overline{g}, \zeta_2 )) ).      
\end{eqnarray*}
\end{lemma}
\begin{proof}
Consider
\begin{eqnarray}
&& \sum_{l \in \Z, \; g\in G_{(l)}} \epsilon^l 
  \F (
 g''_1, z_1; \ldots;
g''_{m+m'}, z_{m+m'}; 
\nn
&& 
\qquad 
 \Po_q (
f (g''_{m+m'+1}, z_{m+m'+1};  \ldots; g''_{m+m' +k}, z_{m+m'+k}); g,\zeta_1))
\nn
&&
   \F (g''_1, z_1; \ldots; 
g''_{m+m'}, z_{m+m'}; 
\nn
&&
 \qquad \Po_q (
 f( g''_{m+m'+k+1}, z_{m+m'+k+1}; \ldots; 
 g''_{m+m'+k+n}, z_{m+m'+k+n});  \overline{g},  \zeta_2)) 
\nn
 && = \sum_{l \in \Z, \;  q\in \Z,  \; g\in G_{(l)}} \epsilon^l    \F ( 
 g''_1, z_1; \ldots;  
 g''_{m+m'}, z_{m+m'};   
\nn
&& \qquad 
  \Po_q (
 f (g''_{m+m'+1}, z_{m+m'+1};  \ldots; g''_{m+m'+ k}, z_{m+m'+k})); g, \zeta_1)  
\nn
  && \F (
g''_1, z_1; \ldots; 
 g''_{m+m'}, z_{m+m'}; 
\nn
&&
  \quad \Po_q ( 
 f(g''_{m+m'+k+1}, z_{m+m'+k+1}; \ldots;  
 g''_{m+m'+k+n}, z_{m+m'+k+n})); \overline{g}, \zeta_2). 
\end{eqnarray}
 Thus, we can rewrite the last expression as 
\begin{eqnarray*}
&&
 \sum_{l \in \Z, \;  q\in \Z, \; g \in G_{(l)} } \epsilon^l \F(
g''_1, z_1; \ldots; 
g''_{m+m'}, z_{m+m'};  
\nn
&&
\qquad   
\; \Po_q(   f (g''_{m+m'+1}, z_{m+m'+1};  \ldots; g''_{m+m'+ k }, z_{m+m'+k}) ); g, \zeta_1 ) 
\nn
& & 
 \qquad   \F (
g''_1, z_1; \ldots;
g''_{m+m'}, z_{m+m'}; 
\nn
&&
\qquad  \qquad 
 \Po_q(  f
(g''_{m+m'+k+1}, z_{m+m'+k+1}; \ldots; g''_{m+m'+k+n}, z_{m+m'+k+n})); \overline{g}, \zeta_2)
\end{eqnarray*}
\begin{eqnarray*}
&&
= \sum_{l \in \Z, \;  q\in \Z, \; g\in G_{(l)} } \epsilon^l \F (
g''_1, z_1; \ldots; 
g''_{m+m'}, z_{m+m'};  
\nn
&&
 \Po_q( f(g''_{m+m'+1}, z_{m+m'+1};  \ldots; g''_{m+m'+ k }, z_{m+m'+k}) 
); g, \zeta_1) 
\nn
&&
 \qquad  \F (
g''_1, z_1; \ldots;
g''_{m+m'}, z_{m+m'}; 
\nn
& & 
 \qquad \qquad   \Po_q( f 
(g''_{m+m'+k+1}, z_{m+m'+k+1}; \ldots; g''_{m+m'+k+n}, z_{m+m'+k+n}; \overline{g}, -\zeta_2)) 
\end{eqnarray*}
\begin{eqnarray*}
&&
= \sum_{l \in \Z, \;  q\in \Z, \; \widetilde{g}' \in G'_{(l)} } \epsilon^l   \F(
g''_1, z_1; \ldots; 
g''_{m+m'}, z_{m+m'};  \widetilde{g}' ) 
\nn
&&
\qquad   
\sum_{l \in \Z, \; g\in G_{(l)} } \epsilon^l 
\; \Po_q (   f(g''_{m+m'+1}, z_{m+m'+1};  \ldots; g''_{m+m'+ k }, z_{m+m'+k}; g, -\zeta_1)) )  
\nn
& & 
\quad   \qquad 
\F (
g''_{1}, z_1; \ldots; 
g''_{m+m'}, z_{m+m'};  \widetilde{g}') 
\nn
&&
\qquad  \Po_q(  f
(g''_{m+m'+k+1}, z_{m+m'+k+1}; \ldots; g''_{m+m'+k+n}, z_{m+m'+k+n}; \overline{g}, -\zeta_2)))  
%
\end{eqnarray*}
\begin{eqnarray*}
&&
= \sum_{q\in \Z}  
  \F (
g''_1, z_1; \ldots; 
g''_{m+m'}, z_{m+m'};  
\nn
&&
\qquad   
 \Po_q(   f(g''_{m+m'+1}, z_{m+m'+1};  \ldots; g''_{m+m'+ k }, z_{m+m'+k}; 
\nn
& & 
\qquad   \qquad 
g''_{m+m'+k+1}, z_{m+m'+k+1}; \ldots; g''_{m+m'+k+n}, z_{m+m'+k+n})).    
\end{eqnarray*}
Now note that as an element of $\M^{k+n+m+m'}_.$, 
\begin{eqnarray*}
&&\F ( 
 g''_1, z_1; \ldots;  
 g''_{m+m'}, z_{m+m'};  
 \Po_q ( f(g''_{m+m'+1}, z_{m+m'+1};  \ldots; 
\nn
&& \qquad
g''_{m+m'+ k }, z_{m+m'+k}; 
 g''_{m+m'+k+1}, z_{m+m'+k+1}; \ldots; g''_{m+m'+k+n}, z_{m+m'+k+n})),
\end{eqnarray*}    
is invariant with respect to the action of $\sigma \in S_{k+n+m+m'}$. 
 Thus we are able to use this invariance to show that the last expression is reduced to 
\begin{eqnarray*}
 && \F ( 
 g''_{k+1}, z_{k+1}; \ldots; g''_{k+1+m}, z_{k+1+m};    
g''_{n+1}, z_{n+1}; \ldots; g''_{n+1+m'}, z_{n+1+m'};    
\nn
&&
 \qquad \Po_q( f(g''_1, z_1;  \ldots; g''_k, z_k;  
 g''_{k+1}, z_{k+1}; \ldots; g''_{k+n}, z_{k+n})))
\nn
&& = \F ( 
 g_{k+1}, x_{k+1};  \ldots; g_{k+1+m}, x_{k+1+m}; 
g'_{n+1}, y_{n+1};  \ldots; g'_{n+1+m'}, y_{n+1+m'};   
\nn 
 && \qquad \Po_q(f(g_1, x_1;  \ldots; g_k, x_k;    
g'_1, y_1; \ldots; g'_n, y_n)).  
\end{eqnarray*}
Similarly, since 
\begin{eqnarray*}
 &&\F (
g''_{1}, z_1; \ldots;
g''_{m+m'}, z_{m+m'};  
\nn
&&
\qquad 
 \Po_q (   
 f (g''_{m+m'+1}, z_{m+m'+1};  \ldots; g''_{m+m' +k}, z_{m+m'+k}; g, \zeta_1)) ), 
\end{eqnarray*}
\begin{eqnarray*}
   &&\F (
g''_1, z_1; \ldots;
g''_{m+m'}, z_{m+m'};  
\nn
&& \qquad 
\Po_q (f (g''_{m+m'+k+1}, z_{m+m'+k+1};  
\ldots; g''_{m+m'+k+n}, z_{m+m'+k+n}); \overline{g}, \zeta_2 ))),   
\end{eqnarray*}
correspond to elements of
 $\M^{m+m'+k}_.$ and 
$\M^{m+m'+k+n}_.$,    
we then obtain 
 $\F(
 g_{k+1}$, $x_{k+1}$; $\ldots$;  
 $g_{k+m}$, $x_{k+m}$;   
 $\Po_q ($  
 $f (g_1, x_1$;  $\ldots$; $g_k$, $x_k$; $g$, $\zeta_1$ $)))$   
and 
   $\F (g'_{n+1}, y_{n+1}$; $ \ldots$;  
 $g'_{n+m'}, y_{n+m'};  
 \Po_q (f(g'_1, y_1; \ldots; g'_n, y_n); \overline{g}, \zeta_2)))$,        
respectively.  
Thus, Lemma follows. 
\end{proof}
\section{Multiplication of $\M^n_m$-elements} 
\label{application}
In this Section we define the multiplication of the spaces 
 $\M^n_m$, $n\ge 0$, $m \ge 0$, 
and coboundary operators $\delta^n_m$ for chain-cochain double 
complexes $(\M^n_m, \delta^n_m)$, and study their properties. 
The matrix element for a number of Lie algebra-valued series represents usually  \cite{TUY} 
 a character associated to a Riemann sphere.
 We extrapolate this notion to the case of $\M^n_m$ spaces. 
 A space $\M^n_m$ 
can be associated with a Riemann sphere with $n$ 
marked points,  
while the multiplication of two such spaces is then associated
 with a sewing of such two spheres with a number of marked 
points 
and extra points with local coordinates identified with formal parameters of
 $\M^k_m$ and $\M^n_{m'}$.  
In order to supply an appropriate geometric construction for the multiplication,  
 we use the $\epsilon$-sewing procedure for two initial spheres
 to obtain a matrix element associated with the multiplication of $\M^n_m$ spaces.  

In our specific case of functions obtained by multiplying elements of 
$\M^._.$-spaces, we take Riemann spheres ${\mathcal S}^{(0)}_a$, $a=1$, $2$, as 
two initial auxiliary
 spaces.    
The resulting 
space is formed by 
the sphere ${\mathcal S}^{(0)}$ obtained by the procedure of sewing ${\mathcal S}^{(0)}_a$. 
The formal parameters $(x_1, \ldots, x_k)$ and $(y_1, \ldots, y_n)$ are identified with 
local coordinates of $k$ and $n$ points on two initial spheres 
${\mathcal S}^{(0)}_a$, $a=1$, $2$ correspondingly.  
In the $\epsilon$ sewing procedure, some $r$ points 
among 
$(p_1, \ldots, p_k)$ 
may coincide with 
points 
among $(p'_1, \ldots, p'_n)$ 
when we identify the annuli.     
This corresponds to the singular case of coincidence of $r$ formal parameters. 

Consider the sphere formed by sewing together two initial 
spheres in the sewing scheme referred to 
as the $\epsilon$-formalism in \cite{Y}. 
Let ${\mathcal S}_a^{(0)}$,   
$a=1$, $2$ 
be 
to initial spheres.  
Introduce a complex sewing
parameter $\epsilon$ where 
 $|\epsilon |\leq \rho_1 \rho_2$,
Consider $k$ distinct points on $p_i \in  {\mathcal S}_1^{(0)}$, $i=1, \ldots, k$, 
with local coordinates $(x_1, \ldots, x_k) \in F_k\C$,  
and distinct points $p_j \in  {\mathcal S}_2^{(0)}$, $j=1, \ldots, n$,
with local coordinates $(y_1, \ldots, y_n)\in F_n\C$,  
with
$\left \vert x_i \right\vert 
\geq |\epsilon |/\rho_2$, 
$\left\vert y_j\right\vert \geq |\epsilon |/\rho_1$. 
Choose a local coordinate $z_a\in \mathbb{C}$ 
on ${\mathcal S}^{(0)}_a$ in the
neighborhood of points $p_a \in{\mathcal S}^{(0)}_a$, $a=1$, $2$. 
Consider the closed disks  
$\left\vert \zeta_a \right\vert \leq \rho_a$, 
 and excise the disk 
$D_a= \{
\zeta_a, \; \left\vert \zeta_a \right\vert  
\leq |\epsilon |/\rho_{\overline{a}}\}\subset {\mathcal S}^{(0)}_a$, 
to form a punctured sphere  
$\widehat{\mathcal S}^{(0)}_a={\mathcal S}^{(0)}_a \backslash \{\zeta_a, \left\vert 
\zeta_a \right\vert \leq |\epsilon |/\rho_{\overline{a}}\}$.
We use the convention 
$\overline{1}=2$, $\overline{2}=1$.    
Define the annulus
$\mathcal{A}_a=\left\{\zeta_a,|\epsilon |/\rho_{\overline{a}}\leq \left\vert
\zeta_a \right\vert \leq \rho_a \right\}\subset \widehat{\mathcal S}^{(0)}_a$,  
and identify $\mathcal{A}_1$ and $\mathcal{A}_2$ as a single region 
$\mathcal{A}=\mathcal{A}_1 \simeq \mathcal{A}_2$ via the sewing relation 
$\zeta_1\zeta_2=\epsilon$.   
In this way we obtain a genus zero compact Riemann surface 
 ${\mathcal S}^{(0)}=\left\{ \widehat {\mathcal S}^{(0)}_1
\backslash \mathcal{A}_1 \right\}
\cup \left\{\widehat{\mathcal S}^{(0)}_2 \backslash 
\mathcal{A}_2 \right\}\cup \mathcal{A}$.  
 We introduce the multiplication of two double complex 
spaces 
with the image in another double complex space coherent with respect 
to the original coboundary operator \eqref{hatdelta}, and the symmetry property \eqref{shushu}.
For $G$-elements $(g_1, \ldots, g_n )$, $ (g'_1, \ldots,  g'_n )\in G$, 
 $\F (g_1, x_1$; $\ldots$ ; $g_k, x_k)$    
$\in \M^k_m$, and        
 $\F( g'_1, y_1; \ldots ; g'_n, y_n)$, $\in \M^n_{m'}$, 
are combined with $m$ and $m'$ $\gamma$-series correspondingly, 
 we introduce the multiplication for $\epsilon=\zeta_1 \zeta_2$,   
$\cdot_\epsilon:  \M^k_. \times \M^n_. \to  \M^{k+n}_.$,    
for $(x_1, \ldots, x_k; y_1, \ldots, y_n) \in F_{k+n}\C$.     
Let us assume that for any $g \in G$, $\F (g_1, x_1; \ldots ; g_k, x_k; g, \zeta_1)\in \M^{k+1}_m$, 
and $\F( g'_1, y_1; \ldots ; g'_n, y_n; \overline{g}, \zeta_2)$, $\in \M^{n+1}_{m'}$,
with $\zeta_1 \zeta_2=\epsilon$, and $\overline{g}$ is dual to $g$ with respect to $(.,.)$.  
The most natural choice of the multiplication supported by the geometrical 
consideration above has the following form 
\begin{eqnarray}
\label{Z2n_pt_eps1q1}
&& \F(g_1, x_1; \ldots; g_k, x_k; g'_1, y_1; \ldots; g'_n, y_n; \epsilon)   
\nn
 &&   =  \sum_{ l\in \mathbb{Z}, \; g \in G_{(l)}} \epsilon^l  
\F (g_1, x_1;  \ldots; g_k, x_k; g, \zeta_1 )\; 
\F (g'_1, y_1; \ldots; g'_n, y_n; \overline{g}, \zeta_2),      
\end{eqnarray}
parameterized by 
$\zeta_1$, $\zeta_2  \in \C$.  
The sum 
is taken over any $G_{(l)}$-basis,   
where $\overline{g}$ is the dual to $g$ with respect to the 
canonical pairing $( . , . )$,   
 with the dual space to $G$.     
By the standard reasoning \cite{Zhu},  
 \eqref{Z2n_pt_eps1q1} does not depend on the choice of a basis of $g \in G_{(l)}$, $l \in \Z$.   
 The definition of a multiplication
 is also supported by Proposition 
\eqref{correl-fn}.  
 In what follows, 
 we will see that, since $g \in G$ and $\overline{g} \in G'$  
are connected  by the sewing condition, $\zeta_1$ and $\zeta_2$ appear in a relation to each other. 
The form 
of the multiplication defined above is natural in terms of the theory of characters in conformal field theory 
\cite{TUY, FMS}.    
 
Let $t$ be the number of common series the mappings 
$\F(g_1, x_1$;  $\ldots$; $g_k, x_k) \in \M^k_m$ and  
$\F(g'_1, y_1; \ldots; g'_n, y_n) \in \M^n_{m'}$,  
are combined with $m$ and $m'$ series. 
 Similar to the case of common formal parameters, this case is 
separately treated with a decrease to $m+m'-t$ of number of combined series.   
Since we assume that $(x_1, \ldots, x_k; y_1, \ldots, y_n)\in F_{k+n}\C$, i.e., 
coincidences of $x_i$ and $y_j$ are  excluded by the definition of $F_{k+n}\C$.  
In what follows, we exclude this case from considerations. 
We define the action of $\partial_l=\partial_{z_l}={\partial}/{\partial_{z_l}}$, $1\le l \le k+n$, 
the differentiation of
$\F(g_1,  x_1;  \ldots;  g_k, x_k$; $g'_1, y_1; \ldots;  g'_n,  y_n; \epsilon)$ 
with respect to the $l$-th entry of 
 $(x_1,  \ldots,   x_k;   y_1, \ldots,   y_n)$  
 as 
follows 
\begin{eqnarray}
\label{Z2n_pt_eps1qdef}
 & &
\partial_l 
\F( g_1,  x_1;  \ldots;  g_k, x_k;   g'_1, y_1; \ldots;  g'_n,  y_n; \epsilon) 
\\
\nonumber
 & &  = \sum_{m\in \mathbb{Z}, g\in G_{(m)} } \epsilon^m  
    \partial^{\delta_{l,i}}_{x_i}  
   \F ( g_1,  x_1;  \ldots; g_k, x_k; g, \zeta_1)  \; 
  \partial^{\delta_{l,j}}_{y_j} 
\F (g'_1,  y_1; \ldots;  g'_n,  y_n;  \overline{g}, \zeta_2).      
\end{eqnarray}
We define the action of the operator $z^K$ on \eqref{Z2n_pt_eps1q1} as 
\begin{eqnarray}
\label{Z2n_pt_eps1qdefkg}
 & &
z^K.
\F( g_1,  x_1;  \ldots;  g_k, x_k;   g'_1, y_1; \ldots;  g'_n,  y_n; \epsilon) 
\\
\nonumber
 & &  = \sum_{m\in \mathbb{Z}, g\in G_{(m)} } \epsilon^m  
   \F ( z^K g_1,  z x_1;  \ldots; z^K g_k, z x_k; z^K g, z \zeta_1)  \; 
\\
&&
\nonumber 
 \qquad \qquad  \F (z^Kg'_1,  z y_1; \ldots;  z^Kg'_n,  z y_n;  z^K \overline{g}, z\zeta_2).      
\end{eqnarray}
We define the
action of an element $\sigma \in S_{k+n}$ on the multiplication of 
$\F(g_1, x_1;  \ldots$; $g_k, x_k)$ $\in \M^k_.$, and 
$\F(g'_1, y_1; \ldots; g'_n, y_n)$ $\in \M^n_.$, as 
\begin{eqnarray}
\label{Z2n_pt_epsss}
&&   \sigma(\F) (g_1, x_1;  \ldots; g_k, x_k; g'_1, y_1; \ldots; g'_n, y_n; \epsilon) 
\\
&&
\;=   \F (g_{\sigma(1)}, x_{\sigma(1)};  \ldots; g_{\sigma(k)}, x_{\sigma(k)}; 
g'_{\sigma(1)}, y_{\sigma(1)}; \ldots; g'_{\sigma(n)}, y_{\sigma(n)}; 
\epsilon) 
\nn 
\nonumber
& & \; =  
\sum_{g\in G_{(m)} }  
\F (g_{\sigma(1)}, x_{\sigma(1)};  \ldots; g_{\sigma(k)}, x_{\sigma(k)}; g, \zeta_1) \; 
\F (g'_{\sigma(1)}, y_{\sigma(1)}; \ldots; g'_{\sigma(n)}, y_{\sigma(n)}; 
\overline{g} , \zeta_2). 
\end{eqnarray}
Now we formulate the main result of this paper 
\begin{proposition}
\label{tolsto}
For $\F(g_1, x_1; \ldots; g_k, x_k) \in \M_m^k$ and 
$\F(g'_1, y_1; \ldots; g'_n, y_n)\in \M_{m'}^n$,  
the multiplication $\F\left(g_1, x_1; \ldots; g_k, x_k; g'_1, y_1; \ldots; g'_n, y_n; \epsilon\right)$ 
\eqref{Z2n_pt_epsss} 
belongs to the space $\M^{k+n}_{m+m'}$, i.e.,  
$\cdot_\epsilon : \M^k_m \times \M_{m'}^n \to  \M_{m+m'}^{k+n}$.   
\end{proposition}
We start from the proof of the convergence of the multiplication of two elements of double complexes 
 to an predetermined meromorphic function defining their multiplication. 
In order to prove convergence 
we have to use a geometrical interpretation \cite{H2, Y}. 
For an infinite-dimensional Lie algebra $\mathfrak g$, the definition of predetermined meromorphic 
functions combined with a number of $G$-series  
 with formal parameters taken as local coordinates on a Riemann sphere. 
Geometrically, each space $\M^n_m$ is associated to a Riemann sphere
 with a few marked points, 
and local coordinates 
 vanishing at these points 
\cite{H2}. 
Two extra points  
can be associated to centers of annuli used in order 
to sew two spheres to form another sphere. 
The multiplication \eqref{Z2n_pt_eps1q1} has then a transparent geometric interpretation
 and associated to a Riemann sphere formed as a result of sewing procedure.  
Let us identify (as in \cite{Zhu}) 
 two sets $(x_1, \ldots, x_k)$ and $(y_1, \ldots, y_n)$ of 
complex formal parameters,
with local
coordinates of two sets of points on the first and the second Riemann spheres correspondingly.   
Identify complex parameters $\zeta_1$, $\zeta_2$ of \eqref{Z2n_pt_eps1q1} 
with coordinates as in $D_a$ of 
the annuluses $\mathcal{A}_a$.  
After identification of annuluses $\mathcal A_a$ and $\mathcal A_{\overline{a}}$, 
$r$ coinciding coordinates may occur. This takes into account case of coinciding formal parameters.
The multiplication is defined 
by a sum of multiplications of matrix elements  associated to each of two spheres.  
Such sum is supposed to describe a predetermined meromorphic function defined 
on a sphere formed 
as a result of geometrical sewing \cite{Y} of two initial spheres. 
Since two initial spaces $\M^k_.$ and $\M^n_.$  
consists of predetermined meromorphic functions,  
it is then proved (Proposition \ref{tolsto}), that the multiplication results in 
a predetermined meromorphic function, i.e., 
an element of the space $\M^{k+n}_{m+m'}$ 
form by means of an absolute convergent matrix element on the resulting sphere. 
The complex sewing parameter, parameterizing the module space of sewing spheres, 
parameterizes also the multiplication of 
$\M^._.$-spaces.  
\begin{proof}
We would like to show that the multiplication \eqref{Z2n_pt_eps1q1} of 
elements of the spaces $\M^k_.$ and $\M^n_.$  
 corresponds to an absolutely converging in $\epsilon$ 
 meromorphic function 
with only possible poles at $x_i=x_j$, $y_{i'}=y_{j'}$, and    
 $x_i=y_{j'}$, $1 \le i, i' \le k$, $1 \le j, j' \le n$.  
In order to prove this we use the geometrical interpretation of the multiplication 
\eqref{Z2n_pt_eps1q1} in terms of Riemann spheres with marked points.  
We consider two sets of $G$-elements $(g_1, \ldots, g_k)$ and $(g'_1, \ldots, g'_k)$, 
and two sets of 
formal complex parameters 
$(x_1, \ldots, x_k)$, $(y_{1}, \ldots, y_n)$. Formal parameters are identified with local coordinates of
$k$ points on the Riemann sphere $\widehat{\mathcal S}^{(0)}_1$, and 
$n$ points on $\widehat{\mathcal S}^{(0)}_2$, 
 with excised annuli ${\mathcal A}_a$.  
Recall the sewing parameter condition $\zeta_1 \zeta_2 = \epsilon$ of the 
 sewing procedure. 
Recall from definition of the disks $D_a$ 
that in two sphere $\epsilon$-sewing formulation,  
the complex parameters $\zeta_a$, $a=1$, $2$ are coordinates inside identified annuluses ${\mathcal A}_a$, and 
$|\zeta_a|\le \rho_a$.  
Therefore, due to Proposition \ref{n-comm} 
the $m$-th coefficients of the expansions in $\zeta_1$ and $\zeta_2$,   
${\mathcal R}_m(x_1,  \ldots, x_k) = 
 {\rm coeff} \; \F (g_1, x_1;  \ldots; g_k, x_k; g, \zeta_1)$ and        
 $\widetilde{\mathcal R}_m(y_1, \ldots, y_n) =
 {\rm coeff} \; \F(g'_1, y_1; \ldots; g'_n, y_n; \overline{g}, \zeta_2)$,     
are absolutely convergent in powers of $\epsilon$ with some radia of convergence $R_a\le \rho_a$, with 
$|\zeta_a|\le R_a$. 
The dependence of the above coefficients on $\epsilon$ is expressed via $\zeta_a$, $a=1$, $2$.  
Let us rewrite the multiplication \eqref{Z2n_pt_eps1q1} as 
\begin{eqnarray}
\label{pihva}
&& 
 \F (g_1, x_1; \ldots; g_k, x_k; g'_1, y_1; \ldots; g'_n, y_n;  
 \epsilon) 
\nn
&&
=  \sum_{l\in \mathbb{Z}, \; g \in G_{(l)}, \; m \in \Z}  
 \epsilon^{l-m-1} \;  {\mathcal R}_m(x_1, \ldots, x_k) \;   
\widetilde{\mathcal R}_m(y_1, \ldots, y_n),        
\end{eqnarray}
  as a formal series in $\epsilon$    
for $|\zeta_a|\leq R_a$, where 
 and  
$|\epsilon |\leq r$ for $r<\rho_1\rho_2$.  
Then we 
 apply 
 Cauchy's inequality to 
 coefficients in $\zeta_1$, $\zeta_2$ above 
 to find 
\begin{equation}
 \left| {\mathcal R}_m
(x_1, \ldots, x_k) \right| \leq {M_1} {R_1^{-m}},  \label{Cauchy1} 
\end{equation}
with 
$M_1=\sup_{ \left| \zeta_1 \right| \leq R_1, \left|\epsilon \right| \leq r} 
\left| {\mathcal R}(x_1, \ldots, x_k ) \right|$.     
Similarly, 
\begin{equation}
 \left|  \widetilde{\mathcal R}_m (y_1, \ldots, y_n) \right| 
\leq {M_2} {R_2^{-m}},  \label{Cauchy2} 
\end{equation}
for
$M_2=\sup_{|\zeta_2|\leq R_2,|\epsilon |\leq r} 
\left|\widetilde{\mathcal R}(y_1, \ldots, y_n)\right|$.     
Using \eqref{Cauchy1} and \eqref{Cauchy2} we obtain for \eqref{pihva}  
\begin{eqnarray*}
 \left| 
 \left( 
 \F (g_1, x_1; \ldots; g_k, x_k; g'_1, y_1; \ldots; g'_n, y_n; \epsilon  
 \right)_l\right| 
  &\le&  
 \left|{\mathcal R}_m (x_1, \ldots, x_k)\right|\;  
   \left|\widetilde{\mathcal R}_m (y_1, \ldots, y_n)\right| 
\nn
   &\le&   
M_1 \; M_2 \;\left(R_1  R_2\right)^{-m}.   
\end{eqnarray*}
Thus, for
 $M=\min\left\{M_1, M_2\right\}$ and $R=\max\left\{R_1, R_2\right\}$,   
\begin{equation}
 \left|
\mathcal R_l(
  x_1;  \ldots, x_k;  
  y_1,  \ldots, y_n)
\right| \leq 
M R^{-l+m+1}.  \label{Cauchy} 
\end{equation}
Due to completeness of 
$C^{k+n}$ and density of the space of meromorphic functions, 
 we see that \eqref{Z2n_pt_eps1q1} is absolute convergent 
to a meromorphic function  
$\F(x_1, \ldots, x_k; y_1, \ldots, y_n; \epsilon)$  
 as a formal series in $\epsilon$   
for $|\zeta_a|\leq \rho_a$, and  
$|\epsilon |\leq r$ for $r<\rho_1\rho_2$, 
 with extra poles 
only at $x_i=y_j$, $1\le i \le k$, $1\le j \le n$. 

Now we prove that multiplication \eqref{Z2n_pt_eps1q1} satisfies the \eqref{ldir1}, \eqref{loconj}.  
 By using \eqref{ldir1} for $\F(g_1,  x_1;  \ldots; g_k, x_k)$
and $\F(g'_1$,  $y_1$; $\ldots$;  $g'_n$,  $y_n)$,  
 we consider 
\begin{eqnarray}
\label{Z2n_pt_eps1q00000}
 & &
\partial_l  \F( g_1,  x_1;  \ldots;  g_k, x_k;   g'_1, y_1; \ldots;  g'_n,  y_n; \epsilon) 
\nn
 & &  = \sum_{m\in \mathbb{Z}\; g\in G_{(m)}} \epsilon^m 
 \partial^{\delta_{l,i}}_{x_i} 
   \F ( g_1,  x_1;  \ldots; g_k, x_k; g, \zeta_1)\;     
 \partial^{\delta_{l,j}}_{y_j} 
\F (g'_1,  y_1; \ldots;  g'_n,  y_n;  \overline{g}, \zeta_2)  
\nn
& &  
=  \sum_{m\in \mathbb{Z}, \; g\in G_{(m)}} \epsilon^m
\F (g_1, x_1; \ldots; T^{\delta_{l,i}} g_i, x_i; \ldots; g_k, x_k; g, \zeta_1) 
\nn
& &
\qquad   \qquad 
\F(g'_1, y_1; \ldots; T^{\delta_{l,j}} g'_j, y_j;  
 \ldots; g'_n, y_n; \overline{g},  \zeta_2) 
\nn
& &  
=    
\F (g_1,  x_1;  \ldots;  T_l. .,. ;  \ldots;  g'_n,  y_n;
\epsilon),    
\end{eqnarray}
where $T_l.$ acts on the $l$-th entry of  
$(g_1,  \ldots;  g_k;   g'_1,  \ldots,   g'_n)$. 
Summing over $l$ we obtain
\begin{eqnarray*}
&&
\sum\limits_{l=1}^{k+n} \partial_l 
\F (g_1, x_1; \ldots; g_k, x_k; g'_1, y_1; \ldots; 
 g'_n, y_n; \epsilon) 
= 
\sum\limits_{l=1}^{k+n} 
\F (g_1, x_1; \ldots; T_l. ., .;   \ldots;  g'_n, y_n; \epsilon) 
\nn
&&
\qquad = T.\F(g_1, x_1;  \ldots ; g_k, x_k;   g'_1,  y_1; \ldots; g'_n,  y_n; \epsilon). 
\end{eqnarray*}
Due to \eqref{loconj} and \eqref{Z2n_pt_eps1qdefkg},   
consider 
\begin{eqnarray*}
&& \F (z^K. (g_1, x_1;  \ldots; g_k, x_k; g'_1, y_1; \ldots; g'_n, y_n; \epsilon))  
\nn
 & &  =\sum_{m \in \mathbb{Z}\; g\in G_{(m)} } \epsilon^m
 \F (z^K g_1, z\; x_1;  \ldots; z^K g_k, z\; x_k; z^K g, z \;\zeta_1)  
\nn
& &
\qquad   \qquad 
\F (z^K g'_1, z \; y_1; \ldots; z^K g'_n, z\; y_n; z^K \overline{g}, z \; \zeta_2)  
\nn
& &  
=  \sum_{m \in \mathbb{Z}, \; g \in G_{(m)} } \epsilon^m 
 z^K. 
\F (g_1, x_1;  \ldots; g_k, x_k; g, \zeta_1) 
\nn
& &
\qquad   \qquad 
 z^K.\F (g'_1, y_1; \ldots; g'_n, y_n; \overline{g}, \zeta_2)  
\nn
& &  
=  z^K.   
\F (g_1, x_1;  \ldots; g_k, x_k; g'_1, y_1; \ldots; g'_n, y_n; \epsilon)  
\end{eqnarray*}
Now we prove compatibility of the multiplication with extra Lie algebra series.  
We will show that $\F\left(
g_1, x_1; \ldots; g_k,  x_k; g'_1, y_1; \ldots; g'_n,  y_n
; \epsilon\right)$
 \eqref{Z2n_pt_epsss}  
is combined with $m+m'$ series.  
Recall that 
$\F(g_1, x_1; \ldots; g_k,  x_k)$
is combined with $m$ series, and 
 $\F(g'_1, y_1; \ldots; g'_n,  y_n)$ 
is combined with $m'$ series. 
For $\F(g_1, x_1; \ldots; g_k,  x_k)$ we have the following. 
Let $(l_1, \dots, l_k) \in \Z_+$ such that $\sum\limits_{i=1}^k l_i= k+m$,  and 
$(g_1, \dots, g_{k+m}) \in G$. Set  
$h_i
=
\F (g_{k_1}, x_{k_1} - \zeta_{i};  
 \ldots; 
g_{k_i}, x_{k_i}- \zeta_i 
 ; \one_G)$,     
where
$k_1=\sum_{j=1}^{i-1}l_j +1$, $ \ldots$, $k_i=\sum_{j=1}^{i-1}l_j +l_i$,  
for $i=1, \dots, k$. 
Then the series 
\begin{eqnarray}
\label{Inms}
\mathcal C^k_m(\F)=
\sum_{(r_1, \dots, r_k) \in \Z} 
\F(\Po_{r_1} h_1, \zeta_1; 
 \ldots; 
\Po_{r_k} h_k, \zeta_k),  
\end{eqnarray} 
is absolutely convergent  when   
\begin{eqnarray}
\label{granizy1}
|x_{l_1+\ldots +l_{i-1}+p}-\zeta_i| 
+ |x_{l_1+\ldots +l_{j-1}+q}-\zeta_j|< |\zeta_i -\zeta_j|, 
\end{eqnarray} 
for $i$, $j=1, \dots, k$, $i\ne j$ and for $p=1, 
\dots,  l_i$ and $q=1, \dots, l_j$. 
There exist positive integers $\beta^k_m(g_i, g_j)$, 
depending only on $g_i$ and $g_j$ for $i, j=1, \dots, k$, $i\ne j$, such that 
the sum is analytically extended to a
meromorphic function
in $(x_1, \dots, x_{k+m})$, 
 independent of $(\zeta_1, \dots, \zeta_k)$,  
with the only possible poles at 
$x_i=x_j$, of order less than or equal to 
$\beta^k_m(g_i, g_j)$, for $i$, $j=1, \dots, k$,  $i\ne j$.  
For $\F(g'_1, y_1; \ldots; g'_n,  y_n)$,  
 let $(l'_1, \dots, l'_n) \in \Z_+$ such that $\sum\limits_{i}^n l'_i= n+m'$,  
$(g'_1, \dots, g_{n+m'}') \in G$.  
Set  
%
$h'_{i'} 
=
f(g'_{k'_1}, y_{k'_1}- \zeta'_{i'};  
 \ldots; 
g'_{k'_{i'}}, y_{k'_{i'}}- \zeta'_{i'})$,      
where
 $k'_1=\sum_{j=1}^{i'-1}l'_j +1$,   $\ldots$, $k'_{i'} =\sum_{j=1}^{i'-1} l'_j +l'_{i'}$,   
for $i'=1, \dots, n$. 
  Then the series 
\begin{eqnarray}
\label{Jnms}
\mathcal C^n_{m'}(\F)=  
\sum_{(r'_1, \dots, r'_n) \in \Z}  
\F(\Po_{r'_1} h'_1, \zeta'_1; 
 \ldots; 
\Po_{r'_n} h'_n, \zeta'_n), 
\end{eqnarray} 
is absolutely convergent  when 
\begin{eqnarray}
\label{granizy2}
|y_{l'_1 + \ldots +l'_{i'-1}+p'}-\zeta'_{i'}| 
+ |y_{l'_1 +\ldots +l'_{j'-1}+q'}-\zeta'_{j'}|< |\zeta'_{i'}
-\zeta'_{j'}|, 
\end{eqnarray} 
for $i'$, $j'=1, \dots, n$, $i'\ne j'$ and for $p'=1, 
\dots,  l'_i$ and $q'=1, \dots, l'_j$. 
There exist positive integers $\beta^n_{m'}(g'_{i'}, g'_{j'})$,  
depending only on $g'_{i'}$ and $g'_{j'}$ for $i'$, $j'=1, \dots, n$, $i'\ne j'$, such that 
the sum is analytically extended to a meromorphic function
in $(y_1, \dots, y_{n+m'})$,   
 independent of $(\zeta'_1, \dots, \zeta'_n)$,   
with the only possible poles at 
$y_{i'}=y_{j'}$, of order less than or equal to 
$\beta^n_{m'}(g'_{i'}, g'_{j'})$, for $i'$, $j'=1, \dots, n$,  $i'\ne j'$.  

Now let us consider the conditions of compatibility for the multiplication 
\eqref{Z2n_pt_epsss} of 
$\F(g_1, x_1; \ldots; g_k,  x_k)$ 
 and 
 $\F(g'_1, y_1; \ldots; g'_n,  y_n)$ combined with a number of series. 
 We redefine the notations 
  $(g''_1$, $\ldots$, $g''_k; g''_{k+1}$, $\ldots$, $g''_{k+m}; g''_{k+m+1}$, $\dots$, $g''_{k+n+m+m'}$;
 $g_{n+1}$, $\ldots$, $g'_{n+m'})$
=$(g_1$, $\ldots$, $g_k; g_{k+1}$, $\ldots$, $g_{k+m}$; $g'_1, \dots, g'_n$; 
 $g'_{n+1}, \ldots, g'_{n+m'}$),  
$(z_1$, $\ldots$, $z_k; z_{k+1}$,  $\ldots$, $z_{k+n})$ =  
  $(x_1$, $\ldots$, $x_k$; $y_1$, $\ldots$, $y_n$),  
 of $G$-elements.  
Introduce $(l''_1, \dots, l''_{k+n}) \in \Z_+$,     
 such that $\sum_{j=1}^{k+n} l''_j = k+n+m+m'$.  
Define  
$h''_i
= f(g''_{k''_1}, z_{k''_1}- \zeta''_{i''};  
 \ldots; 
g''_{k''_{i''}}, z_{k''_{i''}}- \zeta''_{i''})$,      
where
 $k''_1= \sum_{j=1}^{i''-1} l''_j +1, \quad  \ldots, \quad  
k''_{i''}=\sum_{j=1}^{i''-1}l''_j+l''_{i''}$,    
for $i''=1, \dots, k+n$, 
and we take 
$(\zeta''_1, \ldots, \zeta''_{k+n})= (\zeta_1, \ldots, \zeta_k; \zeta'_1, \ldots, \zeta'_n)$. 
Then we consider 
\begin{eqnarray}
\label{Inmdvadva}
 && \mathcal C^{k+n}_{m+m'}(\F)=  
\sum_{r''_1, \dots, r''_{k+n}\in \Z}
\F(\Po_{r''_1 }h''_1, \zeta''_1; 
 \ldots; 
\Po_{r''_{k+n}} h''_{k+n}, \zeta''_{k+n}), 
\end{eqnarray} 
and prove it is absolutely convergent with some conditions. 
 The condition   
$|z_{l''_1+\ldots +l''_{i-1}+p''}$- $\zeta''_i|$  
+ $|z_{l''_1+\ldots +l''_{j-1}+q''}-\zeta''_j|< |\zeta''_i -\zeta''_j|$,   
of absolute convergence for \eqref{Inmdvadva} for $i''$, $j''=1, \dots, k+n$, $i'' \ne j''$ and for $p''=1,  
\dots,  l''_i$ and $q''=1, \dots, l''_j$, follows from the conditions \eqref{granizy1} and \eqref{granizy2}.  
 We obtain 
\begin{eqnarray*}
 && \left|\mathcal C^{k+n}_{m+m'}(\F)\right| 
=\left|\sum_{r''_1, \dots, r''_{k+n}\in \Z}
\F(\Po_{r''_1}h''_1,  \zeta''_1; 
 \ldots; 
\Po_{r''_{k+n}} h''_{k+n}, \zeta''_{k+n})  
\right| 
\nn
&&
=\left| \sum\limits_{l \in \Z, \;  g\in G_{(l)}, \; 
  (r_1, \dots, r_k)\in \Z}  
\F(\Po_{r_1}h_1,  \zeta_1; 
 \ldots; 
\Po_{r_k} h_k, \zeta_k; g, \zeta)   \right.
\nn
&&
\left. 
\qquad \qquad   \sum_{(r'_1, \dots, r'_n)\in \Z} 
\F(\Po_{r'_1}h'_1,  \zeta'_1; 
 \ldots; 
\Po_{r'_n} h'_n, \zeta'_n; \overline{g}, \widetilde{\zeta}) 
 \right|
 \le \left|\mathcal C^k_m(\F)\right| \; \left|\mathcal C^n_{m'}(\F)\right|.  
\end{eqnarray*} 
Thus,  \eqref{Inmdvadva}
is absolutely convergent. 
The maximal orders of possible poles of \eqref{Inmdvadva} are 
$\beta^k_m(g_i, g_j)$, $\beta^n_{m'}(g'_{i'}, g'_{j'})$
 at $x_i=x_j$, $y_{i'}=y_{j'}$. 
In \eqref{Z2n_pt_eps1q1} we obtain an expansion in powers of $x_i$ and $y_j$ we see that new poles 
at $x_i=y_j$ may occur. 
From the last expression we infer that  
there exist 
  positive integers $\beta^{k+n}_{m+m'}(g''_{i''}, g''_{j''})$, 
such that 
 $ \beta^k_m(g_i, g_j) \beta^n_{m'}(g'_{i'}, g'_{j'})
\le \beta^{k+n}_{m+m'}(g''_{i''}, g''_{j''})$, 
 for $i$, $j=1, \dots, k$, $i\ne j$,   $i'$, $j'=1, \dots, n$, $i'\ne j'$, 
depending only on $g''_{i''}$ and $g''_{j''}$ for $i''$,  $j''=1, \dots, k+n$, $i''\ne j''$
such that 
 the series \eqref{Inmdvadva} 
can be analytically extended to a
meromorphic function
in $(x_1, \dots, x_k; y_1, \ldots, y_n)$,    
 independent of $(\zeta''_1, \dots, \zeta''_{k+n})$,    
with extra 
 possible poles at  
 and $x_i=y_j$ of order less than or equal to 
$\beta^{k+n}_{m+m'}(g''_{i''}, g''_{j''})$, for $i''$, $j''=1, \dots, k+n$,  $i''\ne j''$.  

For $\F(g_1, x_1; \ldots; g_k, x_k)  \in  \M^k_m$, 
 the series  
$\mathcal D^k_m(\F)$ =  
$\sum_{q\in \Z}$  
 $\F (g_1$, $x_1$;  $\ldots$;   
 $g_m, x_m$;   
$\Po_q( f(g_{m+1}$, $x_{m+1}$;  $\ldots$;  
$g_{m+k}$, $x_{m+k}$
$)$,  
is absolutely convergent when 
 $x_i \ne x_j$, $i\ne j$,  
$|x_i|>|x_{k'}|>0$, 
 for $i=1, \dots, m$, and $k'=m+1, \dots, k+m$.  
The sum can be analytically extended to a
meromorphic function 
in $(x_1, \dots,  x_{k+m})$ with the only possible poles at 
$x_i=x_j$, of orders less than or equal to 
$\beta^k_m(g_i, g_j)$, for $i, j=1, \dots, k$, $i\ne j$.  
For 
 $\F(g'_1, y_1; \ldots; g'_n, y_n)  \in   \M_{m'}^n$,
 the series 
 $\mathcal D^n_{m'}(\F)$=   
$\sum_{q\in \Z}  \F (g'_1$, $y_1$; $\ldots$;  
 $g'_{m'}$, $y_{m'}$;   
$\Po_q( f($ $g'_{m'+1}$,  $y_{m'+1}$; $\ldots$; $g'_{m'+n}, y_{m'+n})$  
$))$, 
is absolutely convergent when 
$y_{i'}\ne y_{j'}$, $i'\ne j'$, 
$|y_{i'}|>|y_{k''}|>0$,  
 for $i'=1, \dots, m'$, and $k''=m'+1, \dots, n +m'$, and the sum can be analytically extended to a
meromorphic function 
in $(y_1, \ldots, y_{n+m'})$ with the only possible poles at 
$y_{i'}=y_{j'}$, of orders less than or equal to 
$\beta^n_{m'}(g'_{i'}, g'_{j'})$, for $i'$,  $j'=1, \dots, n$, $i'\ne j'$.   
 For the multiplication \eqref{Z2n_pt_epsss}, 
  $(g''_{1}, \dots, g''_{k+n +m+m'})\in G$,  
and 
$(z_1, \ldots $, $ z_{k+n+m+m'} )\in \C$, and  
we find  
positive integers  
$\beta^{k+n}_{m+m'}(g'_i, g'_j)$,  
 depending only on $v'_i$ and 
$v''_j$, for $i''$, $j''=1, \dots, k+n$, $i''\ne j''$.    
Under conditions
$z_{i''}\ne z_{j''}$, $i''\ne j''$, 
$|z_{i''}|>|z_{k'''}|>0$,  
 for $i''=1, \dots, m+m'$, and $k'''=m+m'+1, \dots, m+m'+ k+n$, 
let us introduce 
 $\mathcal D^{k+n}_{m+m'}(\F) = \sum_{q\in \Z}$ 
 $\F (
g''_1, z_1; \ldots; 
g''_{m+m'}, z_{m+m'};  
 \Po_q ( f( g''_{m+m'+1}, z_{m+m'+1}; \ldots; g''_{m+m'+k+n}, z_{m+m'+k+n})) 
; \epsilon)$.    
Using Lemma \ref{obvlem} we then obtain 
\begin{eqnarray*}
&&
\left|\mathcal D^{k+n}_{m+m'}(\F) \right| 
= \left| \sum_{q\in \Z} \F (
g''_1, z_1; \ldots; 
g''_{m+m'}, z_{m+m'};  \right. 
\nn
&&
\left. 
\qquad \Po_q ( f( g''_{m+m'+1}, z_{m+m'+1}; \ldots; g''_{m+m'+k+n}, z_{m+m'+k+n}))
; \epsilon ) \right|
\nn
&&
= \left| \sum_{q\in \Z, \; g\in G} 
   \F (
g_{k+1}, x_{k+1}; \ldots;
g_{k+m}, x_{k+m}; 
\Po_q(  
 f(g_1, x_1;  \ldots; g_k, x_k); g, \zeta_1)) 
\right. 
\nn
&&
 \left. \qquad 
  \F(
g'_{n+1}, y_{n+1}; \ldots; 
g'_{n+m'}, y_{n+m'};  
\Po_q( f (g'_1, y_1; \ldots; g'_n, y_n); \overline{g}, \zeta_2 ) ) ) \right|
\nn
&&
\quad  \le \left|
\mathcal D^k_m(\F) \right| \left|  \mathcal D^n_{m'}(\F)\right|, 
\end{eqnarray*}
where we have used 
the invariance of \eqref{Z2n_pt_epsss} with respect to 
$\sigma \in S_{m+m'+k+n}$. 
In the last expression, 
 according to Proposition \ref{correl-fn} 
$\mathcal D^k_m(\F)$ and $\mathcal D^n_{m'}(\F)$ 
are absolute convergent. 
Thus, $\mathcal D^{k+n}_{m+m'}(\F)$ 
is absolutely convergent, and 
  \eqref{Inmdvadva} is analytically extendable to a meromorphic function  
in $(z_1, \dots, z_{k+n+m+m'})$ with the only possible poles at 
$x_i=x_j$, $y_{i'}=y_{j'}$, and 
at $x_i=y_{j'}$, i.e., the only possible poles at 
$z_{i''}=z_{j''}$, of orders less than or equal to 
$\beta^{k+n}_{m+m'}(v''_{i''}, v''_{j''})$, 
for $i''$, $j''=1, \dots, k'''$, $i''\ne j''$.  

 Finally, for the action of $\sigma \in S_{k+n}$ on the product we have 
\begin{eqnarray*} 
&&
 \sum_{\sigma\in J_{k+n; s}^{-1}}(-1)^{|\sigma|} 
\F(g_{\sigma(1)}, x_{\sigma(1)}; \ldots; g_{\sigma(k)}, x_{\sigma(k)};  
g'_{\sigma(1)},  y_{\sigma(1)};  \ldots; g'_{\sigma(n)},  y_{\sigma(n)}; \epsilon) 
\nn
&&
=
\sum_{\sigma\in J_{k+n; s}^{-1}, \; g\in G_{(l)}}(-1)^{|\sigma|} \epsilon^l 
 \F(g_{\sigma(1)}, x_{\sigma(1)}; \ldots; g_{\sigma(k)},  x_{\sigma(k)}; g, \zeta_1)   
\nn
&&
 \qquad \qquad \qquad  
 \; \F(g'_{\sigma(1)}, y_{\sigma(1)}; \ldots; g'_{\sigma(n)},  y_{\sigma(n)}; \overline{g},  \zeta_2)  
\nn
&&
=\sum\limits_{\; r \in \Z, \; 
 \sigma\in J_{k; s}^{-1}   }\epsilon^r \; (-1)^{|\sigma|} 
\F_{r}(g_{\sigma(1)}, x_{\sigma(1)}; \ldots; g_{\sigma(k)},  x_{\sigma(k)}; \zeta_1)   
\nn
&&
\sum_{r'\in \Z, \;\sigma\in J_{n; s}^{-1}}\epsilon^{r'}  \; (-1)^{|\sigma|} 
\F_{r'}(g'_{\sigma(1)}, y_{\sigma(1)}; \ldots; g'_{\sigma(n)},  y_{\sigma(n)}; \zeta_2) 
 =0,  
\end{eqnarray*}
due to $J^{-1}_{k+n; s}= J^{-1}_{k;s} \times J^{-1} _{n;s}$,  definition \eqref{Z2n_pt_epsss}, 
and 
 $\F_r(g_{\sigma(1)}$, $x_{\sigma(1)}$; $\ldots$; $g_{\sigma(k)}$,  $x_{\sigma(k)}$; $\zeta_1)$  $\in \M^k_m$,     
$\F_{r'}(g'_{\sigma(1)}$, $y_{\sigma(1)}$; $\ldots$; $g'_{\sigma(n)}$,  $y_{\sigma(n)}$; $\zeta_2)$ $ \in \M^n_{m'}$, 
 and, therefore, \eqref{shushu} is satisfied. 
This finishes the proof of the proposition. 
\end{proof}
Let us now recall \cite{Huang} the definition of 
 the coboundary operator 
for the spaces $\M^n_m$, 
\begin{eqnarray}
\label{hatdelta}
  \delta^n_m \F (g_1, z_1; \ldots; g_n, z_n)  
&=& 
\sum_{i=1}^n(-1)^i \; \F\left( \ldots; \gamma_{g_i}( z_i - z_{i+1}) \;  g_{i+1}; \ldots \right)   
\nn
&+& 
  \F (\gamma_{g_1} \left(z_1 \right);  g_2, z_2; \ldots; g_n, z_n)  
\nn
 &+&
 (-1)^{n+1} \F(\gamma_{g_{n+1}}(z_{n+1}); g_1, z_1; \ldots; g_n, z_n). 
\end{eqnarray}
The following lemma takes place:  
\begin{lemma}
\label{cochainprop}
The operator \eqref{hatdelta} obeis 
$\delta^n_m: \M_m^n 
\to \M_{m-1}^{n+1}$,    
$\delta^{n+1}_{m-1} \circ \delta^n_m=0$,  
$0\longrightarrow \M_m^0
\stackrel{  \delta^0_m  }{\longrightarrow} \M_{m-1}^1
\stackrel{  \delta^1_{m-1}}{\longrightarrow}\ldots 
\stackrel{  \delta^{m-1}_1}{\longrightarrow}
 \M_0^m \longrightarrow 0$,  
i.e., provides the double chain-cochain complex $\left(\M_m^n, \delta^n_m \right)$. \hfill $\qed$
\end{lemma}
Then one has 
\begin{corollary}
The multiplication \eqref{Z2n_pt_epsss} extends the chain-cochain 
complex $(\M^n_m, \delta^n_m)$
 to all multiplications $\M^k_m \times \M^n_{m'}$,   
$k$, $n \ge0$, $m$, $m' \ge0$. 
\hfill $\qed$
\end{corollary}

\end{document}